\documentclass[12pt]{amsart}

\usepackage{amsfonts,amsmath,amssymb,color,amscd,amsthm}
\usepackage[T1]{fontenc} 
\usepackage[francais,english]{babel} 
\usepackage{typearea}
\usepackage[T1]{fontenc} 

\usepackage{graphics}

\usepackage[all,arc,curve,color,frame,graph,matrix,cmtip]{xy}

\usepackage[shortlabels]{enumitem}
\setlist[enumerate]{label=\rm{(\arabic*)}}
\setlist[enumerate,2]{label=\rm({\it\roman*})}
\setlist[itemize]{label=\raisebox{0.25ex}{\tiny$\bullet$}}

\usepackage[ colorlinks, linktocpage, citecolor = blue, linkcolor = blue]{hyperref} 

\newcommand\Q{{\mathbb Q}}
\newcommand\R{{\mathbb R}}
\newcommand\A{{\mathbb A}}

\newcommand\C{{\mathbb C}}
\newcommand\F{{\mathbb F}}
\renewcommand\k{\mathrm{k}}
\newcommand\kk{\mathbf{k}}

\newcommand\p{{\mathbb P}}

\newtheorem{theorem}{Theorem}
\newtheorem{lemma}{Lemma}[section]
\newtheorem{corollary}[lemma]{Corollary}
\newtheorem{proposition}[lemma]{Proposition}

\theoremstyle{definition}
\newtheorem{definition}[lemma]{Definition}

\theoremstyle{remark}
\newtheorem{remark}[lemma]{Remark}
\newtheorem{example}[lemma]{Example}

\DeclareMathOperator{\Aut}{Aut}
\DeclareMathOperator{\GL}{GL}
\DeclareMathOperator{\PGL}{PGL}
\DeclareMathOperator{\PSL}{PSL}
\DeclareMathOperator{\SL}{SL}

\DeclareMathOperator{\Bir}{Bir}
\DeclareMathOperator{\Cr}{Cr}

\DeclareMathOperator{\id}{\mathrm{id}}

\title{Topological simplicity of the Cremona groups}

\author{J\'er\'emy Blanc}
\address{J\'er\'emy Blanc, Universit\"{a}t Basel, Spiegelgasse $1$, CH-$4051$ Basel, Switzerland.}
\email{jeremy.blanc@unibas.ch}

\author{Susanna Zimmermann}
\address{Susanna Zimmermann, Universit\"{a}t Basel, Spiegelgasse $1$, CH-$4051$ Basel, Switzerland}
\email{susanna.zimmermann@unibas.ch}

\date{\today}

\thanks{Both authors acknowledge support by the Swiss National Science Foundation Grant  ``Birational Geometry'' PP00P2\_153026.}

\begin{document}
\begin{abstract}The Cremona group is topologically simple when endowed with the Zariski or Euclidean topology, in any dimension $\ge 2$ and over any infinite field. Two elements are moreover always connected by an affine line, so the group is path-connected. \end{abstract}
\maketitle

\section{Introduction}
Fixing a field $\k$ and an integer $n$, the \emph{Cremona group of rank $n$ over $\k$} can be described algebraically as the field automorphism group
$\Cr_n(\k)=\Aut_\k(\k(x_1,\dots,x_n))$
or geometrically as the group
$\Bir_{\p^n}(\k)$
of birational transformations of $\p^n$ that are defined over the field $\k$.

\bigskip

As described in \cite{Se} (see section~\ref{Section:ZarTopCremona} below), one can endow the group with a natural \emph{Zariski topology}, which is induced by \emph{morphisms} $A\to \Bir_{\p^n}$, where $A$ is an algebraic varieties (see $\S\ref{Section:Reminders}$). In \cite{Bl10}, it was shown that the group
$\Bir_{\p^2}(\k)$ was topologically simple when endowed with this topology (i.e.~it does not contain any non-trivial closed normal strict subgroup), when $\k$ is algebraically closed. In this text, we generalise the result and prove the following:

\begin{theorem}\label{TheoremZariski}
For each infinite field $\k$ and each $n\ge 1$, the group $\Bir_{\p^n}(\k)$ is topologically simple when endowed with the Zariski topology $($i.e.~it does not contain any non-trivial closed normal strict subgroup$)$.
\end{theorem}
\begin{remark}
For each field $\k$, the group $\Bir_{\p^2}(\k)$ is not simple as an abstract group \cite{CL13,Lon}. If $\k=\mathbb{R}$, it contains moreover normal subgroups of index $2^m$ for each $m\ge 1$ \cite{Zim}. For each $n\ge 3$ and each field $\k$, deciding wether the abstract group $\Bir_{\p^n}(\k)$ is simple or not is a still wide open question.
\end{remark}
\begin{remark}
If $\k$ is a finite field, the Zariski topology on $\Bir_{\p^n}(\k)$ is the discrete topology (see Lemma~\ref{Lemm:DiscreteFinite}), so the topological simplicity is equivalent to the simplicity as an abstract group, and is therefore false for $n=2$, and open for $n\ge 3$. For $n=1$, this is true if and only if $\k=\F_{2^a}$, $a\ge 2$ (see Lemma~\ref{Lemm:PGL2Finitefield}).
\end{remark}

If $\k$ is a local field, which we assume to be locally compact and nondiscrete (for example $\R$, $\C$, $\F_q((t))$, or a finite extension of $\Q_p$), then there exists a natural topology on $\Bir_{\p^n}(\k)$, which makes it a Hausdorff topological group, and whose restriction on any algebraic subgroup (for example on $\Aut_{\p^n}(\k)=\PGL_{n+1}(\k)$ and $(\PGL_2(\k))^n\subset \Aut_{(\p^1)^n}(\k)$) is the Euclidean topology (the classical topology given by distances of matrices) \cite[Theorem 3]{BF13}. This topology was called \emph{Euclidean topology} of $\Bir_{\p^n}(\k)$. We will show the following analogue of Theorem~\ref{TheoremZariski}, for this topology:
\begin{theorem}\label{TheoremEuclidean}
For each $($locally compact$)$ local field $\k$ and each $n\ge 2$, the topological group $\Bir_{\p^n}(\k)$ is topologically simple when endowed with the Euclidean topology $($i.e.~it does not contain any non-trivial closed normal strict subgroup$)$.
\end{theorem}
\begin{remark}The result is, of course, false for $n=1$, since $\PSL_2(\R)$ is a non-trivial normal strict subgroup of $\PGL_2(\R)$, closed for the Euclidean topology.
\end{remark}

When $\k$ is algebraically closed, the group $\Bir_{\p^n}(\k)$ is connected with respect to the Zariski topology \cite[Th\'eor\`eme 5.1]{Bl10}. We generalise this result as follows:

\begin{theorem}\label{Thm:Connected}
For each infinite field $\k$, each $n\ge 2$ and each $f,g\in \Bir_{\p^n}(\k)$, there is a morphism $\rho\colon \A^1\to \Bir_{\p^n}$, defined over $\k$, such that $\rho(0)=f$ and $\rho(1)=g$. In particular, the group $\Bir_{\p^n}(\k)$ is connected with respect to the Zariski topology $($this second property is also true for $n=1$, even if the first is false$)$.

For each $n\ge 2$, the groups $\Bir_{\p^n}(\R)$ and $\Bir_{\p^n}(\C)$ are path-connected, and thus connected, with respect to the Euclidean topology.
\end{theorem}

The article is organised as follows: 

Section~\ref{Section:Reminders} gives some preliminary results, many of them being classical. In \S\ref{Section:ZarTopCremona}, we recall the notion of families of birational maps, and the topology induced by these ones, for algebraic varieties defined over an algebraically closed field. The natural generalisation to the case when the field is not algebraically closed is provided in \S\ref{Subsec:Notalgclosed}. We then recall, in \S\ref{Subsec:Hd}, the constructions of \cite{BF13} that yield parametrising spaces for birational maps of $\p^n$ of bounded degree. The Euclidean topology on $\Bir(\p^n)$ is then explained in \S\ref{Subsec:Euclideantop}. The last part of the section, \S\ref{Subsec:Projlingroups}, gives some basic results on the structure of $\PGL_n(K)$ and $\PSL_n(K)$, viewed as abstract groups or endowed with the Zariski topology.

The proofs of the theorems are given in Section~\ref{Sec:Results}. We first explain a simple construction of families of conjugates of a birational map corresponding to a local linearisation around a fixed point ($\S\ref{Subssec:Fixed}$). This one is used in \S\ref{Subsec:Closednormal} to study closed normal subgroups of the Cremona groups and to give the proof of Theorems~\ref{TheoremZariski} and~\ref{TheoremEuclidean}. The main steps consist in proving that non-trivial closed normal subgroups contain non-trivial elements of $\Aut_{\p^n}(\k)$ (Proposition~\ref{Prop:LimitconjugatesNormal}) and to show that the normal subgroup generated by $\Aut_{\p^n}(\k)$ is dense (Proposition~\ref{Prop:TheProp}).
 The construction of $\S\ref{Subssec:Fixed}$ is then used in \S\ref{Subsec:Connected} to study the connectedness, and prove Theorem~\ref{Thm:Connected}.

\section{Preliminaries}
\label{Section:Reminders}
\subsection{The families of birational maps and the Zariski topology induced}
\label{Section:ZarTopCremona}
In \cite{De}, M.~Demazure introduced the following functor (that he called $\mathrm{Psaut}$, for pseudo-automorphisms, the name he gave to birational transformations):
\begin{definition}\label{Defi:BirX}
Let $\kk$ be an algebraically closed field, $X$ be an irreduible algebraic variety and $A$ be a noetherian scheme, both defined over $\kk$. We define
\[\begin{array}{lll}
\Bir_X(A)&=&\left\{\begin{array}{l}
A\text{-birational transformations of }A\times X\text{ inducing an}\\
\text{ isomorphism }U\to V,\text{ where }U,V\text{ are open subsets}\\
\text{ of }A\times X,\text{ whose projections on }A\text{ are surjective }\end{array}\right\},\vspace{0.2cm}\\
\Aut_X(A)&=&\left\{\begin{array}{l}
A\text{-automorphisms of }A\times X\end{array}\right\}=\Bir_X(A)\cap \Aut(A\times X).\end{array}\]
\end{definition}
\begin{remark}
When $A=\mathrm{Spec}(\kk)$, we see that $\Bir_X(A)$ corresponds to the group of birational transformations of $X$ defined over $\kk$, that we will write $\Bir_X(\kk)$. Similarly, $\Aut_X(\kk)$ corresponds to the group of automorphisms of $X$ defined over $\kk$.
\end{remark}

Definition~\ref{Defi:BirX} implicitly gives rise to the following notion of families, or morphisms $A\to \Bir_X(\kk)$ (as in \cite{Se,Bl10,BF13}):
\begin{definition}\label{Defi:Morphisms}
Taking $A,X$ as above, an element $f\in \Bir_X(A)$ and a $\kk$-point $a\in A(\kk)$, we obtain an element $f_a\in \Bir_X(\kk)$ given by $x\dasharrow p_2(f(a,x))$, where $p_2\colon A\times X\to X$ is the second projection. 

The map $a\mapsto f_a$ represents a map from $A$ $($more precisely from the $A(\kk)$-points of $A)$ to $\Bir_X(\kk)$, and will be called a \emph{$\kk$-morphism} (or morphism defined over $\kk$) from $A$ to $\Bir_X$. If moreover $f\in \Aut_X(A)$, then $f$ also yields a morphism from $A$ to $\Aut_X$.
\end{definition}

\begin{remark}
If $X,Y$ are two irreducible algebraic varieties and $\psi\colon X\dasharrow Y$ is a birational map (all of them defined over an algebraically closed field $\kk$), the two functors $\Bir_X$ and $\Bir_Y$ are isomorphic, via $\psi$. In other words, morphisms $A\to \Bir_X$ corresponds, via $\psi$, to morphisms $A\to \Bir_Y$. 
The same holds with $\Aut_X$ and $\Aut_Y$, if $\psi$ is an isomorphism.
\end{remark}

Even if $\Bir_X$ is not representable by an algebraic variety or an ind-algebraic variety in general \cite{BF13}, we can define a topology on the group $\Bir_X(\kk)$, given by this functor. This topology was called \emph{Zariski topology} by J.-P. Serre in \cite{Se}:

\begin{definition}  \label{defi: Zariski topology}
Let $X$ be an irreducible algebraic variety defined over $\kk$. A subset~$F\subseteq \Bir_X(\kk)$ is \emph{closed in the Zariski topology}
if for any algebraic variety~$A$ and any $\kk$-morphism~$A\to \Bir_X$ the preimage of~$F$ in $A(\kk)$ is closed.
\end{definition}

\begin{remark}
In this definition, one can of course replace ``any algebraic variety $A$'' with ``any \emph{irreducible} algebraic variety $A$''.
\end{remark}
Endowed with this topology, $\Bir_{\p^n}(\kk)$ is connected, for each $n\ge 1$, and $\Bir_{\p^2}(\kk)$ is topologically simple
for each algebraically closed field $\kk$
\cite{Bl10}.

\subsection{The case where the base field is not algebraically closed}
\label{Subsec:Notalgclosed}
Similarly as for algebraic varieties defined over a non-algebraically closed field $\k$, one can define morphisms defined over $\k$:

\begin{definition}
Let $\k$ be a field, $X$ be a geometrically irreducible (irreducible over the algebraic closure $\kk$ of $\k$) algebraic variety and $A$ be an algebraic variety, both defined over $\k$.

A \emph{$\k$-morphism $A\to \Bir_X$} is a $\kk$-morphism $A\to \Bir_X$ (in the sense of Definition~\ref{Defi:Morphisms}), which is \emph{defined over $\k$}. This latter means that the corresponding $A$-birational map of $A\times X$ is defined over $\k$.
\end{definition}

\begin{remark}
As for morphisms of algebraic varieties, for any field extension $\k\subset \k'$, any $\k$-morphism $A\to \Bir_X$ is also a $\k'$-morphism, and thus yields a map $A(\k')\to \Bir_X(\k')$.
\end{remark}
 We then obtain a natural Zariski topology on $\Bir_X(\k)$, similar to the one of algebraic varieties defined over $\k$:

\begin{definition}  \label{defi: Zariski topologyk}
Let $X$ be an irreducible algebraic variety defined over $\k$. A subset~$F\subseteq \Bir_X(\k)$ is \emph{closed in the Zariski topology}
if for any $\k$-algebraic variety~$A$ and any $\k$-morphism~$A\to \Bir_X$ the preimage of~$F$ in $A(\k)$ is closed.
\end{definition}

Let us make the following observation:
\begin{lemma}\label{Lemm:Stronger}
Let $\k$ be a field, $X$ be a geometrically irreducible algebraic variety defined over $\k$. The Zariski topology on $\Bir_X(\k)$ is finer than the topology on $\Bir_X(\k)$ induced by the Zariski topology of $\Bir_X(\kk)$, where $\kk$ is the algebraic closure of $\k$.
\end{lemma}
\begin{proof}
We show that for each closed subset $F'\subset \Bir_X(\kk)$, the set $F=F'\cap \Bir_X(\k)$ is closed with respect to the Zariski topology.

To do this, we need to show that the preimage of $F$ by any $\k$-morphism~$\rho\colon A\to \Bir_X$ is closed. By definition of the Zariski topology of $\Bir_X(\kk)$, the set $C=\{a\in A(\kk)\mid \rho(a)\in F'\}$ is Zariski closed in $A(\kk)$. The closure $R$ of $C\cap A(\k)$ in $A(\kk)$ is defined over $\k$ \cite[Lemma 11.2.4]{Springer}. Since $R(\kk)\subset C(\kk)$, we have $R\cap A(\k)=R(\k)\subset  C\cap A(\k)\subset R\cap A(\k)$, so $C\cap A(\k)=R(\k)$ is closed in $A(\k)$. 

It remains to observe that the equality $F=F'\cap \Bir_X(\k)$ implies that $C\cap A(\k)=\{a\in A(\k)\mid \rho(a)\in F'\}=\{a\in A(\k)\mid \rho(a)\in F\}=\rho^{-1}(F).$
\end{proof}

\begin{lemma}\label{Lemm:DiscreteFinite}
Let $\k$ be a finite field and $X$ be an algebraic variety defined over $\k$. The Zariski topology on $\Bir_{X}(\k)$ is the discrete topology.
\end{lemma}
\begin{proof}
Let us show that any subset $F\subset \Bir_{X}(\k)$ is closed. For this, we take a $\k$-algebraic variety~$A$, a $\k$-morphism~$\rho\colon A\to \Bir_X$, and observe that $\rho^{-1}(F)$ is finite in $A$, hence is closed.  
\end{proof}

\subsection{The varieties $H_d$}
\label{Subsec:Hd}
The following algebraic varieties are useful to study morphisms to $\Bir_{\p^n}$.
\begin{definition}\label{DefWHG}\cite[Definition 2.3]{BF13}
Let $d,n$ be positive integers.
\begin{enumerate}
\item
We define $W_d$ to be the the projective space parametrising equivalence classes of non-zero  $(n+1)$-uples $(h_0,\dots,h_n)$
of homogeneous polynomials $h_i\in \k[x_0,\dots,x_n]$ of degree $d$,
where $(h_0,\dots,h_n)$ is equivalent to $(\lambda h_0,\dots,\lambda h_n)$ for any $\lambda\in \k^{*}$.
The equivalence class of $(h_0,\dots,h_n)$ will be denoted by $[h_0:\dots:h_n]$.
\item
We define $H_d\subseteq W_d$ to be the set of elements $h=[h_0:\dots:h_n]\in W_d$
such that the rational map
$\psi_h\colon \p^n\dasharrow \p^n$ given by 
\[[x_0:\dots:x_n]\dasharrow
[h_0(x_0,\dots,x_n):\dots:h_n(x_0,\dots,x_n)]\] is birational.
We denote by $\pi_d$ the map $H_d(\k)\to \Bir_{\p^n}(\k)$ which sends $h$ onto~$\psi_h$.
\end{enumerate}
\end{definition}
\begin{proposition}\label{Prop:HdPidKMorphism}
Let $d,n$ be positive integers. Taking the above definitions, the following hold:
\begin{enumerate}
\item
The set $H_d$ is locally closed in the projective space $W_d$ and thus inherits the structure of an algebraic variety;
\item
The map $\pi_d$ corresponds to a morphism $H_d\to \Bir_{\p^n}$, defined over any field. For each field $\k$, the image of the corresponding map $H_d(\k)\to \Bir_{\p^n}(\k)$ consists of all birational maps of degree $\le d$.
\end{enumerate}
\end{proposition}
\begin{proof}
Follows from \cite[Lemma 2.4]{BF13}.
\end{proof}

\subsection{The Euclidean topology}\label{Subsec:Euclideantop}
Suppose now that $\k$ is a local field, which is moreover locally compact and nondiscrete.

The Euclidean topology of $\Bir_{\p^n}(\k)$ described in \cite[Section 5]{BF13} is as follows: we put on $W_d(\k)\simeq \p^n(\k)$ the classical Euclidean topology, then put on $H_d(\k)\subset W_d(\k)$ the induced topology, and put on $\pi_d(H_d(\k))=\{f\in \Bir_{\p^n}(\k)\mid \deg(f)\le d\}$ the quotient topology induced by $\pi_d$. The Euclidean topology on $\Bir_{\p^n}(\k)$ is then the inductive limit topology induced by the inclusions 
$$\{f\in \Bir_{\p^n}(\k)\mid \deg(f)\le d\}\to \{f\in \Bir_{\p^n}(\k)\mid \deg(f)\le d+1\}.$$

\begin{lemma}\label{Lemm:MorphismContinuousEuclidean}
Let $\k$ be a local field $($locally trivial and nondiscrete$)$, let $A$ be an algebraic variety defined over $\k$, and let $\rho\colon A\to \Bir_{\p^n}$ be a $\k$-morphism. Then, the map 
$$A(\k)\to \Bir_{\p^n}(\k)$$
is continuous, for the Euclidean topologies.
\end{lemma}
\begin{proof}
There exists an open affine covering $(A_i)_{i\in I}$ of $A$ (with respect to the Zariski topology) such that for each $i\in I$, there exists an integer $d_i$ and a morphism of algebraic varieties $\rho_i\colon A_i\to H_{d_i}$, such that the restriction of $\rho$ to $A_i$ is $\pi_{d_i}\circ \rho_i$ \cite[Lemma 2.6]{BF13}. It follows from the construction that the $A_i,\rho_i$ can be assumed to be defined over~$\k$.

We take a subset $U\subset \Bir_{\p^n}(\k)$, open with respect to the Euclidean topology, and want to show that $\rho^{-1}(U)\subset A(\k)$ is open with respect to the Euclidean topology. As all $A_i(\k)$ are open in $A(\k)$, it suffices to show that $\rho^{-1}(U)\cap A_i(\k)$ is open in $A_i(\k)$ for each $i$. This latter follows from the fact that $\rho|_{A_i}=\pi_{d_i}\circ \rho_i$ and that both are continuous for the Euclidean topology.
\end{proof}
%
%
%
%
%
%
%

\subsection{The projective linear group}\label{Subsec:Projlingroups}
Note that $\Bir_{\p^n}(\k)$ contains the algebraic group $\Aut_{\p^n}(\k)=\PGL_{n+1}(\k)$, and that the restriction of the Zariski topology to this subgroup corresponds to the usual Zariski topology of the algebraic variety $\PGL_{n+1}(\k)$, which can be viewed as the open subset of $\p^{(n+1)^2-1}(\k)$, complement of the hypersurface given by the vanishing of the determinant.

Let us make the following two observations: 
\begin{lemma}\label{Lemm:PSLdenseZariski}
 If $\k$ is an infinite field and $n\ge 2$, then $\PSL_n(\k)$ is dense in $\PGL_n(\k)$ with respect to the Zariski topology. Moreover, every non-trivial normal subgroup of $\PGL_n(\k)$ contains $\PSL_n(\k)$. In particular, $\PGL_n(\k)$ does not contain any non-trivial normal strict subgroup, closed for the Zariski topology.
\end{lemma}
\begin{proof}
$(1)$ Observe that the group homomorphism $\det\colon \GL_n(\k)\to \k^*$ yields a group homomorphism 
\[\det\colon \PGL_n(\k)\to (\k^*)/\{f^n\mid f\in \k^*\},\]
whose kernel is the group $\PSL_n(\k)$. We consider the morphism 
\[\begin{array}{ccc}
\rho\colon \A^1(\k)\setminus \{0\} &\to &\PGL_n(\k)\\
t & \mapsto & \left(\begin{array}{cc} t & 0\\
0 & I\end{array}\right)\end{array}\]
where $I$ is the identity matrix of size $(n-1)\times (n-1)$, and observe that $\rho^{-1}(\PSL_n(\k))$ contains $\{t^n\mid t \in \A^1(\k)\}$, which is an infinite subset of $\A^1(\k)$, and is therefore dense in $\A^1(\k)$. The closure of $\PSL_n(\k)$ contains thus $\rho(\A^1(\k)\setminus\{0\})$. As every element of $\PGL_n(\k)$ is equal to some $\rho(t)$ modulo $\PSL_n(\k)$, we obtain that $\PSL_n(\k)$ is dense in $\PGL_n(\k)$.

$(2)$ Let $N\subset \PGL_n(\k)$ be a normal subgroup with $N\not=\{\id\}$, and let $f\in N$ be a non-trivial element. Since the center of $\PGL_n(\k)$ is trivial, one can replace $f$ with $\alpha f \alpha^{-1} f^{-1}$, where $\alpha\in \PGL_n(\k)$ does not commute with $f$, and assume that $f\in N\cap \PSL_n(\k)$. Then, as $\PSL_n(\k)$ is a simple group \cite[Chapitre II, $\S 2$]{bib:Dieudonne}, we obtain $\PSL_n(\k)\subset N$.  This implies,  that $\PGL_n(\k)$ does not contain any non-trivial normal strict subgroup which is closed for the Zariski topology. 
\end{proof}
\begin{remark} Lemma~\ref{Lemm:PSLdenseZariski} does not work for the Euclidean topology. For instance, for each $n\ge 1$, the group $\PSL_{2n}(\R)$ is a normal strict subgroup of  $\PGL_{2n}(\R)$, that is closed for the Euclidean topology.
\end{remark}
\begin{lemma}\label{Lemm:PGL2Finitefield}
Let $\k$ be a finite field. 

\begin{enumerate}
\item\label{PGL2PSL2}
$\PGL_2(\k)=\PSL_2(\k)$ if and only if $\mathrm{char}(\k)=2$;
\item\label{PGL2simple}
$\PGL_2(\k)$ is a simple group if and only if $\k=\F_{2^a}$, $a\ge 2$.
\end{enumerate} 
\end{lemma}
\begin{proof}
\ref{PGL2PSL2}: As explained before, $\PSL_2(\k)=\PGL_2(\k)$ if and only if every element of $\k^*$ (or equivalently of $\k$) is a square. As $\k$ is finite, the group homomorphism
\[\begin{array}{rcl}
\k^*&\to & \k^{*}\\
x &\mapsto & x^2\end{array}\]
is surjective if and only if it is injective, and this latter corresponds to ask that the characteristic of $\k$ is $2$.

\ref{PGL2simple}: If $\mathrm{char}(\k)\not=2$, then $\PSL_2(\k)\subsetneq\PGL_2(\k)$ is a non-trivial normal subgroup.

If $\mathrm{char}(\k)=2$, $\PGL_2(\k)=\PSL_2(\k)$ is a simple group if and only if $\k\not=\F_2$ (\cite[Chapitre II, $\S 2$]{bib:Dieudonne}).
\end{proof}

\section{Proof of the results}\label{Sec:Results}
\subsection{The construction associated to fixed points}\label{Subssec:Fixed}
Let us explain the following simple construction that we will use often in the sequel.

\begin{example}\label{ExampleConjugates}
Let $f\in \Bir_{\p^n}(\k)$ be an element that fixes the point $p=[1:0:\dots:0]$, and that induces a local isomorphism at $p$.

In the chart $x_0=1$, we can write $f$ locally as 
\[\begin{array}{l}
x=(x_1,\dots,x_n)\dasharrow\left(\dfrac{p_{1,1}(x)+\dots+p_{1,m}(x)}{1+q_{1,1}(x)+\dots+q_{1,m}(x)},\dots,\dfrac{p_{n,1}(x)+\dots+p_{n,m}(x)}{1+q_{n,1}(x)+\dots+q_{n,m}(x)}\right),\end{array}\]
where the $p_{i,j},q_{i,j}\in \k[x_1,\dots,x_n]$ are homogeneous of degree $j$. For each $t\in \k\setminus\{0\}$, the element 
\[\theta_t\colon (x_1,\dots,x_n)\mapsto (tx_1,\dots,tx_n)\]
extends to a linear automorphism of $\p^n(\k)$ that fixes $p$. Then, the map $t\mapsto (\theta_t)^{-1} \circ f \circ \theta_t$ gives rise to a morphism $F\colon \A^1\setminus \{0\}\to \Bir_{\p^n}(\k)$ whose image contains only conjugates of $f$ by linear automorphisms.

Writing $F$ locally, we can observe that $F$ extends to a morphism $\A^1\to \Bir_{\p^n}(\k)$ such that $F(0)$ is linear. Indeed, $F(t)$ can be written locally as  follows:
 \[\begin{array}{l}
 F(t)(x)=F(t)(x_1,\dots,x_n)=\\
\left(\dfrac{p_{1,1}(x)+tp_{1,2}(x)+\dots+t^{m-1}p_{1,m}(x)}{1+tq_{1,1}(x)+\dots+t^mq_{1,m}(x)},\dots,\dfrac{p_{n,1}(x)+tp_{n,2}(x)+\dots+t^{m-1}p_{n,m}(x)}{1+tq_{n,1}(x)+\dots+t^mq_{n,m}(x)}\right),\end{array}\]
and $F(0)$ corresponds to the derivative (linear part) of $F$ at $p$, which is locally given by 
\[(x_1,\dots,x_n)\mapsto \left(p_{1,1}(x),\dots,p_{n,1}(x)\right)\]
and which is an element of $\Aut_{\p^n}(k)\subset \Bir_{\p^n}(\k)$ since $f$ was chosen to be a local isomorphism at $p$.
\end{example}

Using the example above, one can construct $\k$-morphisms $\A^1\to \Bir_{\p^n}$.

\begin{proposition}\label{Prop:Limitconjugates}
Let $\k$ be a field, $n\ge 1$, let $g\in \Bir_{\p^n}(\k)$ and $p\in \p^n(\k)$ be a point such that $g$ induces a local isomorphism at $p$ and fixes $p$.
 Then there exist $\k$-morphisms $\nu\colon \A^1\setminus \{0\}\to \Aut_{\p^n}$ and $\rho\colon \A^1\to \Bir_{\p^n}$ such that the following hold:
 
\begin{enumerate}
\item
For each field extension $\k\subset \k'$ and each $t\in \A^1(\k')\setminus\{0\}$, we have \[\rho(t)=\nu(t)^{-1}\circ g\circ \nu(t).\]
Moreover, $\nu(1)=\id$, so $\rho(1)=g$.
\item
The element $\rho(0)$ belongs to $\Aut_{\p^n}(\k)$. It is the identity if and only if the action of $g$ on the tangent space $T_p(\p^2)$ is trivial. 
\end{enumerate}
\end{proposition}
\begin{proof}
Conjugating by an element of $\Aut_{\p^n}(\k)$, we can assume that $p=[1:0:\dots:0]$. We then choose $\nu$ to be given by 
\[\nu(t)\colon [x_0:x_1:\dots:x_n]\mapsto [x_0:tx_1:\dots:tx_n],\]
and define $\rho\colon \A^1\setminus \{0\}\to \Bir_{\p^n}$ by $\rho(t)=\nu(t)^{-1}\circ g\circ \nu(t)$. As it was shown in Example~\ref{ExampleConjugates}, the $\k$-morphism $\rho$ extends to a $\k$-morphism $\A^1\to \Bir_{\p^n}$ such that $\rho(0)\in \Aut_{\p^n}(\k)$. Moreover, this element is trivial if and only if the action of $g$ on the tangent space $T_p(\p^n)$ is trivial.
\end{proof}

\subsection{Closed normal subgroups of the Cremona groups}
\label{Subsec:Closednormal}
As a consequence of Proposition~\ref{Prop:Limitconjugates}, we obtain the following result:

\begin{proposition}\label{Prop:LimitconjugatesNormal}
Let $\k$ be an infinite field, $n\ge 1$ and $N\subset \Bir_{\p^n}(\k)$ be a normal subgroup, with $N\not=\{\id\}$. If $N$ is closed with respect to the Zariski topology or to the Euclidean topology $($if $\k$ is a local field$)$, then $N\cap \Aut_{\p^n}(\k)$ is not the trivial group.
\end{proposition}
\begin{proof}
We can assume that $n\ge 2$, as the result is trivial for $n=1$ (in which case $\Bir_{\p^n}(\k)=\Aut_{\p^n}(\k)$) .
Let us choose a non-trivial element $f\in N$. As $f$ is a birational transformation, it induces an isomorphism $U\to V$, where $U,V\subset \p^n$ are two non-empty open subsets defined over $\k$. Since $\k$ is infinite, $U(\k)$ and $V(\k)$ are not empty, so we can find $p\in U(\k)$, and $q=f(p)\in V(\k)$. We can moreover choose $p\not=q$, since $\{p\in U\mid f(p)\not=p\}$ is open in $U$, and not empty. Let us take an element $\alpha\in \Aut_{\p^n}(\k)$ that fixes $p$ and $q$. The element $g=\alpha^{-1} f^{-1}\alpha f\in N$  fixes $p$ and is a local isomorphism at this point. We can choose $\alpha$ such that the derivative $D_p(g)$ of $g$ at this point is not trivial, since 
\[D_p(g)=D_{p}(\alpha^{-1})\circ D_{q}(f^{-1})\circ D_q(\alpha)\circ D_p(f).\]
By Proposition~\ref{Prop:Limitconjugates}, there exists a $\k$-morphism $\rho\colon \A^1\to \Bir_{\p^n}$ such that $\rho(0)\in \Aut_{\p^n}(\k)\setminus \{\id\}$ and such that $\rho(t)\in N$ for each $t\in \A^1(\k)\setminus \{0\}$. Since $N$ is closed (with respect to the Zariski or to the Euclidean topology), $\rho^{-1}(N)\subset \A^1(\k)$ is closed (with respect to the Zariski or to the Euclidean topology respectively, see Lemma~\ref{Lemm:MorphismContinuousEuclidean} in the latter case) and contains $\A^1(\k)\setminus \{0\}$. For the Zariski topology, one uses the fact that $\k$ is infinite to get $\rho^{-1}(N)=\A^1(\k)$. For the Euclidean topology, one uses the fact that $\k$ is nondiscrete to get the same result. In each case, we find that $\rho(0)\in N\cap \Aut_{\p^n}(\k)$.
\end{proof}

\begin{lemma}\label{Lem:Autinnormal}
Let $\k$ be an infinite field, $n\geq 2$ an integer and $N\subset \Bir_{\p^n}(\k)$ be a normal subgroup, with $N\cap \Aut_{\p^n}(\k)\not=\{\id\}$. Then $\PGL_{n+1}(\k)=\Aut_{\p^n}(\k)\subset N.$
\end{lemma}
\begin{proof}
The group $N\cap \Aut_{\p^n}(\k)$ is a non-trivial normal subgroup of $\Aut_{\p^n}(\k)=\PGL_{n+1}(\k)$, so contains $\PSL_{n+1}(\k)$ (Lemma~\ref{Lemm:PSLdenseZariski}).

We choose $a\in \k^*$, take $g_a\in N$, $h\in \Bir_{\p^n}(\k)$ given by 
\[\begin{array}{llllll}
g_a\colon &[x_0:\dots:x_n]&\mapsto& [x_0:ax_1:\frac{1}{a} x_2:x_3:\dots:x_n]\\
h\colon&[x_0:\dots:x_n]&\dasharrow &[x_0:x_1:x_2\cdot \frac{x_1}{x_0}:x_3:\dots:x_n].\end{array}\]
Then, $g_a'=hg_ah^{-1}\in N$ is given by 
\[\begin{array}{llllll}
g_a'\colon &[x_0:\dots:x_n]&\mapsto& [x_0:ax_1:x_2:x_3:\dots:x_n].\end{array}\]
As every element of $\PGL_n(\k)$ is equal to some $g_a'$ modulo $\PSL_{n+1}(\k)$, we obtain that $\PGL_{n+1}(\k)\subset N$.
\end{proof}
\begin{proposition}\label{Prop:TheProp}
Let $\k$ be an infinite field, $n\geq 2$ an integer and consider $\Bir_{\p^n}(\k)$ endowed with the Zariski topology or the Euclidean topology $($if $\k$ is a local field$)$. Then the normal subgroup of $\Bir_{\p^n}(\k)$ generated by $\Aut_{\p^n}(\k)$ is dense in $\Bir_{\p^n}(\k)$. 

In particular, $\Bir_{\p^n}(\k)$ does not contain any non-trivial  closed normal strict subgroup. 
\end{proposition}
\begin{proof}$(1)$
Let $f\in\Bir_{\p^n}(\k)$, $f\neq\id$. It induces an isomorphism $U\rightarrow V$, where $U,V\subset\p^n$ are two non-empty open subsets, defined over $\k$. Since $\k$ is infinite, we can find $p\in U(\k)$. There exist $\alpha_1,\alpha_2\in\Aut_{\p^n}(\k)$ such that $g:=\alpha_1f\alpha_2$ fixes $p$, is a local isomorphism at this point and that $D_p(g)$ is not trivial. By Proposition~\ref{Prop:Limitconjugates}, there exist $\k$-morphisms $\nu\colon\A^1\setminus\{0\}\rightarrow\Aut_{\p^n}(\k)$ and $\rho_1\colon\A^1\rightarrow\Bir_{\p^n}(\k)$ such that $\rho_1(t)=\nu(t)^{-1}\circ g^{-1}\circ\nu(t)$ for each $t\in \A^1(\k)\setminus \{0\}$ and $\rho_1(0)\in\Aut_{\p^n}(\k)$. We define a $\k$-morphism
\[\rho_2\colon\A^1\rightarrow\Bir_{\p^n}(\k),\quad \rho_2(t)=\alpha_1^{-1}\circ g\circ \rho_1(t)\circ\rho_1(0)^{-1}\circ\alpha_2^{-1}.\]
Since $\alpha_1,\alpha_2,\rho_1(0),\nu(t)\in\Aut_{\p^n}(\k)$ for all $t\in\A^1\setminus\{0\}$, the map 
\[\rho_2(t)=\alpha_1^{-1}\circ\left(g\circ\nu(t)^{-1}\circ g^{-1}\right)\circ\nu(t)\circ\rho_1(0)^{-1}\circ\alpha_2^{-1}\] is contained in the normal subgroup of $\Bir_{\p^n}(\k)$ generated by $\Aut_{\p^n}(\k)$, for each $t\in\A^1\setminus\{0\}$. Therefore, $f=\rho_2(0)$ is contained in the closure of the normal subgroup of $\Bir_{\p^n}(\k)$ generated by $\Aut_{\p^n}(\k)$.

$(2)$ Let $\{\id\}\neq N\subset\Bir_{\p^n}(\k)$ be a closed normal subgroup (with respect to the Zariski or to the Euclidean topology). It follows from Proposition~\ref{Prop:LimitconjugatesNormal} and Lemma~\ref{Lem:Autinnormal} that $\Aut_{\p^n}(\k)\subset N$. Since $N$ is closed, it contains the closure of the normal subgroup generated by $\Aut_{\p^n}(\k)$, which is equal to $\Bir_{\p^n}(\k)$.
\end{proof}

Note that Proposition~\ref{Prop:TheProp}, together with Lemma~\ref{Lemm:PSLdenseZariski} (for the dimension $1$ in the case of the Zariski topology), yields Theorems~\ref{TheoremZariski} and~\ref{TheoremEuclidean}. 

\subsection{Connectedness of the Cremona groups}\label{Subsec:Connected}
The group $\Bir_{\p^n}$ is connected, for the Zariski topology \cite{Bl10}. More precisely, we have the following:
\begin{proposition}\label{Thm51ofBl10} \cite[Th\'eor\`eme 5.1]{Bl10}
Let $\kk$ be an algebraically closed field and $n\ge 1$. For each $f,g\in \Bir_{\p^n}(\kk)$, there is an open subset $U\subset \A^1(\kk)$ that contains $0$ and $1$, and a morphism $\rho\colon U\to \Bir_{\p^n}(\kk)$ such that $\rho(0)=f$ and $\rho(1)=g$.
\end{proposition}
This corresponds to say that $\Bir_{\p^n}(\kk)$ is ``rationally connected''. We will generalise this for any field $\k$, and provide a morphism from the whole $\A^1$ (Proposition~\ref{Prop:Connectenessfamilies} below), showing then that $\Bir_{\p^n}(\k)$ is ``$\A^1$-uniruled''.\\

Let us recall the following classical fact.
\begin{lemma}\label{Lemm:MorphismAmSLn}
For each field $\k$ and each integer $n\ge 2$, there is an integer $m$ and a $\k$-morphism $\rho\colon \A^m\to \SL_n$ such that $\rho(\A^m(\k))=\SL_n(\k)$.
\end{lemma}
\begin{proof}
Using Gauss-Jordan elimination, every element of $\SL_n(\k)$ is a product of a diagonal matrix and  $r$ elementary matrices of the first kind: matrices of the form $I+\lambda e_{i,j}$, $\lambda\in \k$, $i\not=j$, where $(e_{i,j})_{i,j=1,\dots,n}$ is the canonical basis of the vector space of $n\times n$-matrices. Moreover, the number $r$ can be chosen to be the same for all elements of $\SL_n(\k)$. We then observe  that 
\[\left(\begin{array}{cc}
1&\lambda-1\\
0&1\end{array}\right)
\left(\begin{array}{cc}1&0\\
1& 1\end{array}\right)
\left(\begin{array}{cc}
1&\lambda^{-1}-1\\
0&1
\end{array}\right)
\left(\begin{array}{cc}
1&0\\
-\lambda& 1
\end{array}\right)=\left(\begin{array}{cc}
\lambda &0\\
0& \lambda^{-1}
\end{array}\right)\]
for each $\lambda\in \k^*$. Using finitely many such products, we obtain then all diagonal elements. This gives the existence of $s\in \mathbb{N}$, only dependent on $n$, such that every element of $\SL_n(\k)$ is a product of $s$ elementary matrices of the first kind.

Denoting by $\nu_{i,j}\colon\A^1\to \SL_n(\k)$ the  $\k$-morphism sending $\lambda$ to $I+\lambda e_{i,j}$, this shows that every element of $\SL_n(\k)$ is in the image of a product morphism $\A^m\to \SL_n(\k)$ of finitely many $\nu_{i,j}$. The number of such maps being finite, we can enlarge $m$ and obtain one morphism for all maps.
\end{proof}
\begin{corollary}\label{Coro:PSLn}
For each field $\k$, each integer $n\ge 2$ and all $f,g\in \PSL_n(\k)$, there exists a $\k$-morphism $\nu\colon \A^1\to \PSL_n$ such that $\nu(0)=f$ and $\nu(1)=g$.
\end{corollary}
\begin{proof}
It suffices to take a morphism $\rho\colon \A^m\to \SL_n$ as in Lemma~\ref{Lemm:MorphismAmSLn}, to choose $v,w\in \A^m(\k)$ such that $\rho(v)=f$, $\rho(w)=g$ in $\PSL_n(\k)$, and to define $\nu(t)=\rho(v+t(w-v))$.
\end{proof}
\begin{remark}
By construction, Corollary~\ref{Coro:PSLn} also works for $\SL_n(\k)$, but is in fact false for $\GL_n(\k)$. Indeed, every $\k$-morphism $\nu \colon \A^1\to \GL_n$ gives rise to a morphism $\det\circ \nu \colon \A^1\to \A^1\setminus \{0\}$, necessarily constant. As every morphism $\A^1\to \PGL_n$ lifts to a morphism $\A^1\to \GL_n$, the same problem holds for $\PGL_n$.
\end{remark}

\begin{example}\label{Example:Twoderivatives}
Let $\k$ be a field, $n\ge 2$ and $\lambda\in \k^*$. We consider $g\in \Bir_{\p^n}(\k)$ given by
$$g\colon [x_0:\dots:x_n]\mapsto \left[\frac{x_0(x_1+\lambda x_{2})+x_1x_2}{x_1+x_2}:x_1:\dots:x_n\right]$$
We observe that $p_1=[0:1:0:\dots:0]$ and $p_2=[0:0:1:0:\dots:0]$ are both fixed by~$g$. In local charts $x_1=1$ and $x_2=1$, the map $g$ becomes:
\[\begin{array}{rclll}
\ [x_0:1:x_2:x_3:\dots:x_n]&\mapsto& \left[\dfrac{x_0(1+\lambda x_2 )+x_2}{x_2+1}:1:x_2:x_3:\dots:x_n\right]\vspace{0.1cm}\\
\ [x_0:x_1:1:x_3:\dots:x_n]&\mapsto& \left[\dfrac{x_0(x_1+\lambda)+x_1}{x_1+1}:x_1:1:x_3:\dots:x_n\right]\end{array}\]
Applying Proposition~\ref{Prop:Limitconjugates} to the two fixed points, we then get two $\k$-morphisms $\rho_1,\rho_2\colon \A^1\to \Bir_{\p^n}$ such that $\rho_1(1)=g=\rho_2(1)$ and $\rho_1(0),\rho_2(0)\in \Aut_{\p^n}(\k).$ The two elements are provided by the construction Example~\ref{ExampleConjugates}. Choosing for this one the affine coordinates $x_1\not=0$ and $x_2\not=0$, using permutations of the coordinates, yield the following maps, corresponding to the linear parts in these affine spaces:
\[\begin{array}{rrclll}
\rho_1(0)\colon &[x_0:x_1:x_2:x_3:\dots:x_n]&\mapsto& \left[x_0+x_2:x_1:x_2:x_3:\dots:x_n\right],\vspace{0.1cm}\\
\rho_2(0)\colon &[x_0:x_1:x_2:x_3:\dots:x_n]&\mapsto& \left[x_0\lambda+x_1:x_1:x_2:x_3:\dots:x_n\right].\end{array}\]
\end{example}
\bigskip

We can now give the following generalisation of  \cite[Th\'eor\`eme 5.1]{Bl10} (Proposition~\ref{Thm51ofBl10}):

\begin{proposition}\label{Prop:Connectenessfamilies}
For each infinite field $\k$, each integer $n\ge 2$ and all $f,g\in \Bir_{\p^n}(\k)$, there exists a $\k$-morphism $\nu\colon \A^1\to \Bir_{\p^n}$ such that $\nu(0)=f$ and $\nu(1)=g$.
\end{proposition}
\begin{proof}
Multiplying the morphism with $f^{-1}$, we can assume that $f=\id$. We then denote by $N\subset \Bir_{\p^n}(\k)$ the subset given by
\[N=\left\{
g\in \Bir_{\p^n}(\k)\left| \begin{array}{ll}\text{ there exists a $\k$-morphism }\nu\colon \A^1\to \Bir_{\p^n}\\\text{ such that }\nu(0)=\id\text{ and }\nu(1)=g\end{array}\right.\right\}.\]
Observe that $N$ is a normal subgroup of $\Bir_{\p^n}(\k)$ that contains $\PSL_{n+1}(\k)$ by Corollary~\ref{Coro:PSLn}. As $N$ is a priori not closed, we cannot apply Theorem~\ref{TheoremZariski}. However, we will apply Proposition~\ref{Prop:Limitconjugates} and Example~\ref{Example:Twoderivatives} to obtain the result.

First, taking $\lambda,g,\rho_1,\rho_2$ as in Example~\ref{Example:Twoderivatives}, the morphisms $t\mapsto \rho_i(t)\circ \rho_i(0)^{-1}$, $i=1,2$, show that $g\circ (\rho_1(0))^{-1}, g\circ (\rho_2(0))^{-1}\in N$, which implies that $\rho_1(0)\circ (\rho_2(0))^{-1}\in N$. Since $\rho_1(0)\in \PSL_{n+1}(\k)\subset N$, this implies that 
\[\rho_2(0)\colon [x_0:x_1:x_2:x_3:\dots:x_n]\mapsto \left[x_0\lambda+x_1:x_1:x_2:x_3:\dots:x_n\right]\]
belongs to $N$, for each $\lambda\in \k^*$. Hence, $\Aut_{\p^n}(\k)=\PGL_{n+1}(\k)\subset N$.

Second, we take any $g\in \Bir_{\p^n}(\k)$ of degree $d\ge 2$, take a point $p\in \p^n(\k)$ such that $g$ induces a local isomorphism at $p$, choose $\alpha\in \PSL_{n+1}(\k)$ such that $\alpha\circ g$ fixes $p$. Proposition~\ref{Prop:Limitconjugates} yields then the existence of a $\k$-morphism $\rho\colon \A^1\to \Bir_{\p^n}$, such that $\rho(1)=\alpha\circ g$ and $\rho(0)\in \Aut_{\p^n}(\k)$. Choosing $\rho'\colon \A^1\to \Bir_{\p^n}$ given by $\rho'(t)= \rho(t)\circ \rho(0)^{-1}$, we obtain that
$\rho'(1)=\alpha\circ  g \circ \rho(0)^{-1}\in N$. Since $\alpha,\rho(0)\in \Aut_{\p^n}(\k)\subset N$, this shows that $g\in N$ and concludes the proof.\end{proof}
\begin{corollary}\label{Cor:ConnectedZariski}
For each infinite field $\k$ and each $n\ge 1$, the group $\Bir_{\p^n}(\k)$ is connected, for the Zariski topology.
\end{corollary}
\begin{proof}
For $n=1$, the result follows from the fact that $\Bir_{\p^1}=\Aut_{\p^1}=\PGL_2$ is an open subvariety of $\p^3$. For $n\ge 2$, this follows from Proposition~\ref{Prop:Connectenessfamilies}.
\end{proof}
\begin{corollary}\label{Cor:Pathconnected}
For each $n\ge 2$, the groups $\Bir_{\p^n}(\R)$ and $\Bir_{\p^n}(\C)$ are path-connected, and thus connected, for the Euclidean topology.
\end{corollary}
\begin{proof}
Let us fix $\k=\R$ or $\k=\C$. For each $f,g\in \Bir_{\p^n}(\k)$, there is a $\k$-morphism $\nu\colon \A^1\to \Bir_{\p^n}$ such that $\nu(0)=f$ and $\nu(1)=g$ (Proposition~\ref{Prop:Connectenessfamilies}). The corresponding map $\k=\A^1(\k)\to \Bir_{\p^n}(\k)$ is then continuous, for the Euclidean topologies (Lemma~\ref{Lemm:MorphismContinuousEuclidean}). The restriction of this map to the interval $[0,1]\subset \R\subset \C$ yields a map $[0,1]\to \Bir_{\p^n}(\k)$, continuous for the Euclidean topologies, sending $0$ to $f$ and $1$ to $g$.
\end{proof}

Theorem~\ref{Thm:Connected} is now proven, as a consequence of Proposition~\ref{Prop:Connectenessfamilies} and Corollaries~\ref{Cor:ConnectedZariski} and~\ref{Cor:Pathconnected}.

\end{document}